\documentclass{article}
\usepackage[utf8]{inputenc}

\title{On cancellative pairs of families of subsets}

\usepackage{amsmath, amssymb, amsthm, url, tikz, verbatim, tkz-graph}
\usepackage{blkarray}
\usetikzlibrary{arrows.meta}

\newtheorem{theorem}{Theorem}[section]

\newtheorem{lemma}[theorem]{Lemma}

\newtheorem{corollary}[theorem]{Corollary}

\newtheorem{question}[theorem]{Question}

\newtheorem{conjecture}[theorem]{Conjecture}

\usetikzlibrary{calc}
\usetikzlibrary{patterns}
\usetikzlibrary{decorations.markings}
\usetikzlibrary{arrows,shapes.geometric,%
	decorations.pathreplacing,shapes,shadows}

\date{}

\author{ Yijia Fang\thanks{Department of Mathematics, National University of Singapore. Email: fangyijia@u.nus.edu.}
\and
Hao Huang\thanks{Department of Mathematics, National University of Singapore. Email: huanghao@nus.edu.sg. Research supported in part by a start-up grant at NUS and an MOE Academic Research Fund (AcRF) Tier 1 grant A-8003627.}}

\begin{document}

\maketitle

\begin{abstract}
A pair $(\mathcal{A}, \mathcal{B})$ of families of subsets of $[n]$ is cancellative if whenever $A, A' \in \mathcal{A}, B \in \mathcal{B}$ satisfy $A \cup B=A' \cup B$, then $A=A'$, and whenever $A \in \mathcal{A}, B, B' \in \mathcal{B}$ satisfy $A \cup B=A \cup B'$, then $B=B'$. We show that for every cancellative pair $(\mathcal{A}, \mathcal{B})$, the inequality $|\mathcal{A}||\mathcal{B}| \le 2.25^n$ holds, matching Tolhuizen's $(2.25-o(1))^n$ lower bound construction.

\end{abstract}

\section{Introduction}
The notion of a cancellative pair was introduced by Holzman and K\"orner \cite{HK95}: a pair $(\mathcal{A}, \mathcal{B})$ of families of subsets of $[n]$ is \textit{cancellative} if whenever $A, A' \in \mathcal{A}, B \in \mathcal{B}$ satisfy $A \cup B=A' \cup B$, then $A=A'$, and whenever $A \in \mathcal{A}, B, B' \in \mathcal{B}$ satisfy $A \cup B=A \cup B'$, then $B=B'$. Equivalently, $A \setminus B=A' \setminus B$ implies $A=A'$, and $B \setminus A=B' \setminus A$ implies $B=B'$.

A closely related and stronger concept is the \textit{recovering} pair: $A \setminus B=A'\setminus B'$ implies $A=A'$, $B \setminus A=B' \setminus A'$ implies $B=B'$.  For example, let $\mathcal{A}$ consist of all subsets of $\{1, \cdots, k\}$, and $\mathcal{B}$ consists of all subsets of $\{k+1, \cdots, n\}$. This gives a recovering  pair since $A \setminus B=A$, $B \setminus A=B$. This pair has $|\mathcal{A}||\mathcal{B}|=2^n$. Ahlswede and Simonyi \cite{AS94} conjectured $|\mathcal{A}||\mathcal{B}| \le 2^n$ for every recovering pair $(\mathcal{A}, \mathcal{B})$. A recovering pair is always cancellative but the converse does not hold. On the other hand, when $n=3$, it is not hard to verify that $\mathcal{A}=\mathcal{B}=\{\{1\},\{2\},\{3\}\}$ provides a cancellative but not recovering pair with $|\mathcal{A}||\mathcal{B}|=9>8=2^3$. Taking products of this example gives cancellative pairs $(\mathcal{A}, \mathcal{B})$ with $|\mathcal{A}||\mathcal{B}|=3^{2n/3} \approx 2.08^n$ when $n$ is divisible by $3$. Erd\H os and Katona \cite{Katona74} conjectured that in the symmetric case when $\mathcal{A}=\mathcal{B}$, this is the best construction. This was disproved by Shearer \cite{Shearer96} by constructing symmetric cancellative pairs with $|\mathcal{A}||\mathcal{B}| \approx 2.138^n$. Tolhuizen \cite{Tol2000} found an elegant construction that shows there exist symmetric cancellative pairs with $|\mathcal{A}||\mathcal{B}| = (2.25-o(1))^n$. An earlier result of Frankl and F\"uredi \cite{FF84} showed that this constant is best possible in the symmetric case when $\mathcal{A}=\mathcal{B}$.

For the non-symmetric cancellative pairs, Holzman and K\"orner \cite{HK95} showed in 1995 that $|\mathcal{A}||\mathcal{B}| \le 2.3264^n$. Little progress was made over the following two decades, until Janzer \cite{Janzer18} improved the bound to $2.2682^n$ in 2018. Both bounds also apply to recovering pairs. For recovering pairs specifically, Solt\'esz \cite{Soltesz18} improved the Holzman--K\"orner bound to $2.284^n$. This was subsequently improved to $2.2663^n$ by  Nair and Yazdanpanah \cite{NY20}, and further to $2.2543^n$ by Mond, Souza, and Versteegen \cite{MSV25} using a connection with the constant arising in the study of cancellative pairs. To summarize, if we denote by $\mu_{can}$ and $\mu_{rec}$ the optimal growth rate for cancellative and recovering pairs respectively, formally,
$$\mu_{\mathrm{can}}=\lim_{n \rightarrow \infty} \left(\max_{(\mathcal{A}, \mathcal{B}) \textup{~cancellative~on~} [n]} |\mathcal{A}||\mathcal{B}|\right)^{1/n},$$
$$\mu_{\mathrm{rec}}=\lim_{n \rightarrow \infty} \left(\max_{(\mathcal{A}, \mathcal{B}) \textup{~recovering~on~} [n]} |\mathcal{A}||\mathcal{B}|\right)^{1/n},$$
then the state of the art is 
$$2.25 \le \mu_{can} \le 2.2682, \qquad 2 \le \mu_{rec} \le 2.2543.$$

In this note we use the entropy method to prove the following result, thereby settling the cancellative pair problem:

\begin{theorem}\label{thm_main}
If $(\mathcal{A}, \mathcal{B})$ forms a cancellative pair of families of subsets of $[n]$, then 
$$|\mathcal{A}||\mathcal{B}| \le 2.25^n.$$
\end{theorem}
This upper bound, combined with Tolhuizen's $(2.25-o(1))^n$ construction, shows $\mu_{can}=2.25$. Applying Theorem \ref{thm_main} to the symmetric case \(\mathcal A=\mathcal B\) recovers the exact exponential bound \(|\mathcal A|\le(3/2)^n\). This bound also follows from the earlier Frankl--F\"uredi estimate  $|\mathcal{A}| \le n \cdot 1.5^n$ by a tensor-product argument.
Using a connection $\mu_{rec} \le \max\{2.2499, 2.222^{0.27} \cdot \mu_{can}^{0.73}\}$ established by Mond, Souza, and Versteegen \cite{MSV25}, we obtain $\mu_{rec} < 2.25=\mu_{can}$, creating a positive separation between these two optimal growth rates, and breaking the $2.25$ barrier for $\mu_{rec}$ for the first time.

\section{The proof}
Throughout this paper, we use standard notations of entropy from $\cite{AS92, Galvin14}$: given a random variable $X$, we denote by $H[X]$ its Shannon entropy $H[X]=-\sum_x \mathbb{P}[X=x]\log_2 \mathbb{P}[X=x]$. We also define conditional entropy $H[X|Y]$ in the standard way. Denote by $h(x)$ the function $-x\log_2(x)-(1-x)\log_2(1-x)$, which is the Shannon entropy of the Bernoulli distribution.

We start by stating the following theorem of Sinkhorn \cite{Sinkhorn64}. A special case of it says if $K$ is a square matrix with strictly positive elements, then there exist diagonal matrices $D_1$ and $D_2$ with strictly positive diagonal elements such that $D_1 K D_2$ is doubly stochastic. For general non-square matrices, the following holds:
\begin{lemma}[Sinkhorn's Theorem]
Given an arbitrary $m \times n$ matrix $K$ with strictly positive entries, one can find positive $r_1, \cdots, r_m$ and $s_1, \cdots, s_n$ such that the new matrix $K'_{ij}=K_{ij}r_is_j$ has all row sums equal to $\frac{1}{m}$ and all column sums equal to $\frac{1}{n}$.
\end{lemma}

Here is another technical lemma we need for our entropy proof.

\begin{lemma}\label{lem_key}
Suppose $M$ and $U$ are random variables both taking value in $\{0, 1\}$, and they satisfy: 
$$\mathbb{P}[M=1|U=1]=1,~~~~\mathbb{P}[M=1|U=0]=\frac{1}{3}.$$ 
Suppose $R$ is an arbitrary finite random such that $R\to U \to M$ forms a Markov chain, that is,
$$
\mathbb{P}[M=m|U=u]=\mathbb{P}[M=m|U=u,R=r]
$$
for any $m,u,r$ with $\mathbb{P}[R=r]>0$. Then we always have $$H[M]-H[M|R]+H[U|R, M] \le \log_2 \frac{3}{2}.$$
\end{lemma}
\begin{proof}
For each $r$ such that $\mathbb{P}[R=r]=p_r>0$, we set $a_r=\mathbb{P}[U=0|R=r]$. Then 
$$\mathbb{P}[M=0|R=r]=\mathbb{P}[M=0|U=0] \cdot a_r = \frac{2a_r}{3}.$$
We also have $\mathbb{P}[M=0]=\sum_r \frac{2a_r p_r}{3}$. So 
$$H[M|R]=\sum_r p_r h\left(\frac{2a_r}{3}\right),~~~H[M]=h\left(\sum_r\frac{2a_rp_r}{3}\right).$$ 
To compute $H[U|R, M]$, chain rule gives
\begin{align*}
    H[U|R,M]&=H[U,R,M]-H[R,M]\\
            &=H[U,R,M]-H[U,R]+H[U,R]-H[M,R]\\
            &=H[M|U,R]+H[U|R]-H[M|R].
\end{align*}
Recall that $\mathbb{P}[R=r]=p_r$ and $\mathbb{P}[U=0|R=r]=a_r$, by definition
$$
    H[M|U,R]=\sum_r p_r a_r h\left(\frac{1}{3}\right),~~~H[U|R]=\sum_r p_r h(a_r).
$$
Hence $$H[U|R,M]=\sum_r p_r \left(h(a_r)+a_rh\left(\frac{1}{3}\right)-h\left(\frac{2a_r}{3}\right)\right).$$
To prove the desired inequality, it suffices to show that
$$h\left(\sum_r\frac{2a_rp_r}{3}\right)+\sum_r p_r \left(h(a_r)+a_rh\left(\frac{1}{3}\right)-2h\left(\frac{2a_r}{3}\right)\right) \le \log_2 \frac{3}{2}.$$
We first show that for any $x \in [0,1]$, $g(x):=h(x)+xh(1/3)-2h(\frac{2x}{3})+2x/3 \le 0$. It is not hard to see that $g(x)=h(x)+(\log_2 3) x-2h(\frac{2x}{3})$. Using
$$h'(x)=\log_2\frac{1-x}{x}, \qquad h''(x)=-\frac{1}{(\ln2)x(1-x)},$$
we obtain $g''(x)=\frac{1-2x}{(\ln2)x(1-x)(3-2x)}.$
Thus $g$ is convex on $[0,\frac{1}{2}]$ and concave on $[1/2,1]$. On $[0,1/2]$, we have $g(0)=0$ and $g\left(\frac12\right)
=\frac73-\frac32\log_2 3<0$,
where the last inequality follows from $3^9=19683 \ge 16384=2^{14}$. By convexity, for $0\le x\le  \frac{1}{2}$,
$$g(x)\le (1-2x)g(0)+2x g\left(\frac12\right)\le0.$$
On $[\frac{1}{2},1]$, direct calculations give $g\left(\frac34\right)=0$
and
\begin{align*}
g'\left(\frac34\right)=\log_2\frac{1/4}{3/4}+\log_2 3
-\frac43\log_2\frac{3-2 \cdot 3/4}{2 \cdot 3/4}=0.
\end{align*}
Since $g$ is concave on this interval, we have for $x \in [1/2, 1]$, $g(x) \le g(3/4) =0.$

Applying this inequality to $a=a_r$ and averaging with weight $p_r$, we have 
\begin{align} \label{ineq1}
\sum_r p_r\left(h(a_r)+a_rh\left(\frac{1}{3}\right)-2h\left(\frac{2a_r}{3}\right)\right) \le -\sum_r \frac{2a_rp_r}{3}.
\end{align}

Finally for $x \in [0,1]$, if we let $f(x)=h(x)-x$, then $f'(x)=\log_2 \frac{1-x}{x}-1$. So $f'(x)=0$ iff $x=1/3$, that gives $f(1/3)=h(1/3)-1/3=\log_2 \frac{3}{2}$ which is greater than $f(0)$ and $f(1)$. Therefore $f(x) \le \log_2 \frac{3}{2}.$ This gives
\begin{align}\label{ineq2}
h\left(\sum_r\frac{2a_rp_r}{3}\right) - \sum_r\frac{2a_rp_r}{3} \le \log_2 \frac{3}{2}.
\end{align}
Adding \eqref{ineq1} and \eqref{ineq2} completes the proof of this lemma.
\end{proof}

Now we are ready to prove our main result.

\begin{proof}[Proof of Theorem \ref{thm_main}]
We apply Sinkhorn's Theorem to the matrix $K$ whose rows and columns correspond to the subsets from $\mathcal{A}$ and $\mathcal{B}$ respectively, and $K_{A, B}=3^{-|A \cup B|}$. Suppose the positive weights from Sinkhorn's Theorem are $\{r_A\}_{A \in \mathcal{A}}$ and $\{s_B\}_{B \in \mathcal{B}}$ respectively. 

This naturally defines random variables $X, Y$, such that 
$$\mathbb{P}[X=A, Y=B]=r_A s_B \cdot 3^{-|A \cup B|}.$$
$$\mathbb{P}[X=A]=\frac{1}{|\mathcal{A}|},~~~~\mathbb{P}[Y=B]=\frac{1}{|\mathcal{B}|}.$$
In other words, we modify the joint distribution of $(X, Y)$ such that $X$ is uniform on $\mathcal{A}$, $Y$ is uniform on $\mathcal{B}$, yet they are not necessarily independent. From the uniformity, we have 
$$H[X]=\log_2 |\mathcal{A}|, \qquad H[Y]=\log_2|\mathcal{B}|.$$

Let $U=X \cup Y$. We further introduce a random variable $M$ as follows: $M$ always contains $U=X  \cup Y$ as a subset, and for each element $i \in [n] \setminus U$, we independently add it to $M$ with probability $\frac{1}{3}$. Therefore, for $C \supset A \cup B$,
$$\mathbb{P}[M=C|X=A, Y=B]=\left(\frac{1}{3}\right)^{|C|-|A \cup B|}\left(\frac{2}{3}\right)^{n-|C|}.$$
In other words,
$$\mathbb{P}[M=C, X=A, Y=B]=r_A s_B \cdot 3^{-n} \cdot 2^{n-|C|}\cdot 1_{A \cup B \subset C}.$$
It is not hard to see that $1_{A \cup B \subset C}=1_{A \subset C} \cdot 1 _{B \subset C}$, so the right hand side factors into the product of a function only in $A$ and another function only in $B$. Therefore conditioning on $M=C$ when $\mathbb{P}[M=C]>0$, we have that the events $\{X=A\}$ and $\{Y=B\}$ are independent. As a consequence, the conditional entropies satisfy $H[X|M]+H[Y|M]=H[X,Y|M]$, expanded as
\begin{align}\label{eq1}
H[X,M]+H[Y,M]=H[X,Y,M]+H[M].
\end{align}

Our goal is to prove the following two inequalities:
$$H[M]+H[Y|X,M]-H[M|X] \le n \log_2 \frac{3}{2}, \qquad H[M]+H[X|Y,M]-H[M|Y] \le n \log_2 \frac{3}{2}.$$
Assuming that both are true, their sum gives
\begin{align*}
2n \log_2 \frac{3}{2} &\ge 2H[M]+2H[X,Y,M]-H[X,M]-H[Y,M]-H[X,M]-H[Y,M]+H[X]+H[Y]\\
&=H[X]+H[Y] = \log_2 |\mathcal{A}|+\log_2|\mathcal{B}|.
\end{align*}
The first equality follow from \eqref{eq1}. The last equality comes from that $X$ and $Y$ are uniform. This would immediately give $|\mathcal{A}||\mathcal{B}| \le \left(\frac{9}{4}\right)^n,$ completing the proof. 

Since $(\mathcal{A}, \mathcal{B})$ is a cancellative pair, knowing $U=X \cup Y$ and $Y$ immediately gives $X$, knowing $U$ and $X$ immediately gives $Y$. And obviously knowing both $X$ and $Y$ determines $U$. Therefore $$H[X,Y,M]=H[X,Y,U,M]=H[U,X,M]=H[U,Y,M].$$
Therefore $H[Y|X,M]=H[U|X,M]$ and $H[X|Y,M]=H[U|Y,M]$. It suffices to prove the following two inequalities:
$$H[M]+H[U|X,M]-H[M|X] \le n \log_2 \frac{3}{2}, \qquad H[M]+H[U|Y,M]-H[M|Y] \le n \log_2 \frac{3}{2}.$$

Now we prove the first inequality without using the cancellative property. Then by symmetry the second inequality follows immediately from replacing $X$ by $Y$. We first break down the left-hand side coordinate-wise. We view $M$ as a random binary vector (the indicating vector of $M$ as a subset) in $\{0, 1\}^n$. Let $M=(M_1, \cdots, M_n)$. Similarly write $U=(U_1, \cdots, U_n)$ Then 
\begin{align*}
H[U|X, M]&=\sum_{i=1}^n H[U_i|U_1, \cdots, U_{i-1}, X, M] \le \sum_{i=1}^n H[U_i|X, M_1, \cdots, M_{i-1}, M_i].
\end{align*}
The last inequality is due to the fact that dropping conditioning on $U_1 \cdots, U_{i-1}, M_{i+1}, \cdots, M_n$ would not decrease the entropy. We also use the sub-additivity and chain rule to get
$$H[M] \le \sum_{i=1}^n H[M_i].$$
$$H[M|X]=\sum_{i=1}^n H[M_i|X, M_1, \cdots, M_{i-1}].$$
It suffices to prove for every $i=1, \cdots, n$,
$$H[M_i]-H[M_i|X, M_1, \cdots, M_{i-1}]+H[U_i|X, M_1, \cdots, M_{i-1}, M_i] \le \log_2 \frac{3}{2}.$$
Let $R=(X, M_1, \cdots M_{i-1})$. The inequality becomes
$$H[M_i]-H[M_i|R]+H[U_i|R, M_i] \le \log_2 \frac{3}{2}.$$
Now note that $M_i$ and $U_i$ are both $1$-dimensional binary variables. By our definition of $M$ from $U=X \cup Y$, clearly we have $R\to U_i\to M_i$ forming a Markov chain, and
$$\mathbb{P}[M_i=1|U_i=1]=1,~~~~\mathbb{P}[M_i=1|U_i=0]=\frac{1}{3}.$$
Therefore $M_i, U_i, R$ satisfy the assumptions of Lemma \ref{lem_key}. This completes the proof.
\end{proof}

\section{Tolhuizen's construction}
Here let us briefly describe Tolhuizen's elegant lower bound construction \cite{Tol2000} that achieves $|\mathcal{A}||\mathcal{B}|=(2.25-o(1))^n$. Indeed he constructed a symmetric pair which means a family $|\mathcal{A}|=(1.5-o(1))^n$ such that for $A, B \in \mathcal{A}$, knowing $A \cup B$ and $A$ uniquely determines $B$. For ease of our readers, we use the linear algebra language instead of the original coding theory language. We will choose a $k \times n$ matrix $M$ over $\mathbb{F}_2$ with $k=n/3$. For $|S|=k$, denote by $M_S$ the $k \times k$ submatrix of $M$ consisting of the columns indexed by $S$. We denote by $\sigma(S)$ the sum of columns of $M_S$, which is a vector in $\mathbb{F}_2^k$. For a fixed vector $w \in \mathbb{F}_2^k$, let 
$$\mathcal{F}_{M,w}=\left\{S \in \binom{[n]}{k}: M_S\textup{~is~invertible~and~}\sigma(S)=w\right\}.$$
Tolhuizen's symmetric pair is $\mathcal{A}=\mathcal{B}=\mathcal{F}_{M,w}$ for a carefully chosen matrix $M$ and and a most popular vector $w$. 

To see that this is a cancellative pair, suppose $A \cup B=A \cup C$ for $A,B,C \in \mathcal{F}_{M,w}$. Then $B \Delta C \subset A$. By the definition of $\mathcal{F}_{M,w}$, $\sigma(B)=\sigma(C)$ means the sum of columns in $B \Delta C$ is the zero vector. However $M_A$ is invertible, so the only possibility is $B \Delta C=\emptyset$, meaning $B=C$.

To maximize $|\mathcal{F}_{M,w}|$, first choose entries of  $M$ independently $0$ or $1$ each with probability $1/2$. Then for a fixed $|S|=k$, 
$$\mathbb{P}[M_S\textup{~is~invertible}]=\prod_{j=1}^k (1-2^{-j}) \rightarrow \gamma \approx 0.288788~~~~\textup{when~}k \rightarrow \infty.$$
By linearity of expectation, there exists $M$ such that the number of $k$-subsets $S$ such that $M_S$ is invertible is at least $\gamma \binom{n}{k}$. Then by pigeonhole principle, there exists $w \in \mathbb{F}_2^k$ such that the resulting $\mathcal{F}_{M,w}$ has size at least $\gamma 2^{-k} \binom{n}{k}$. Taking $k=n/3$ optimizes this quantity and gives 
$$|\mathcal{A}|=|\mathcal{B}|=|\mathcal{F}_{M,w}|\ge \gamma 2^{-n/3}\binom{n}{n/3}= (1+o(1))\gamma \cdot \frac{3}{2\sqrt{\pi n}}\cdot 1.5^n=(1.5+o(1))^n.$$

A very natural question is: why cannot one carefully choose an explicit $k \times n$ matrix $M$ such that there is still a constant proportion of $|S|=k$ with $M_S$ being invertible, and among these $S$, the column sum $\sigma(S)$ distributes rather unevenly? If this could happen, then one may choose the most popular $w$ that would give a large $|\mathcal{F}_{M,w}|$.

From our Theorem \ref{thm_main} that shows $|\mathcal{A}||\mathcal{B}| \le 2.25^n$ for all cancellative pairs, or the result of Frankl and F\"uredi that shows for symmetric cancellative pair $\mathcal{A}=\mathcal{B}$, $|\mathcal{A}| \le n \cdot 1.5^n$, we immediately know that this is not possible. The following result describes an anti-concentration phenomenon. Roughly speaking, it says that if most of the $k \times k$ submatrices of $M$ are invertible, then the sum of their column vectors has no atom larger than a subexponential factor times the uniform mass.

\begin{corollary}
For every fixed $k \times n$ matrix $M$ over $\mathbb{F}_2$ with $k=n/3$, let $S$ be a random $k$-subset of $[n]$, and $c=\mathbb{P}[M_S\textup{~is~invertible}]$. Then for every $w \in \mathbb{F}_2^k$ we have 
$$\mathbb{P}[\sigma(S)=w| M_S\textup{~is~invertible}] \le \left(\frac{2\sqrt{\pi n}}{3c}+o(1)\right)\cdot 2^{-k}.$$
\end{corollary}
\begin{proof}
Let the probability on the left hand side be $p$, we know that $pc\binom{n}{k}=|\mathcal{F}_{M, w}| \le 1.5^n$. Estimating the binomial coefficient using Stirling's formula finishes the proof.
\end{proof}

\section{Concluding Remarks}
In this paper we have found the optimal growth rate of cancellative pairs of famillies. For the symmetric case, let $G(n)$ be the maximum of $|\mathcal{A}|$ for $\mathcal{A} \subset 2^{[n]}$ and $(\mathcal{A}, \mathcal{A})$ cancellative. We know $G(n) \le 1.5^n$. It would be interesting to know if the maximum is always attained by a linear algebraic construction. Computer search suggests that this is always the case for $n \le 8$. 
\begin{question}
Is it true that $G(n)$ is always attained by $|\mathcal{F}_{M, w}|$, for some carefully chosen $k \times n$ matrix $M$ over $\mathbb{F}_2$ and vector $w \in \mathbb{F}_2^k$?
\end{question}

Interestingly, computer search shows that for $n=4$, the optimal construction for non-symmetric cancellative pairs is $\mathcal{A}=\{\{1\},\{2\},\{3\}\}$ and $\mathcal{B}=\{\{1\},\{2\},\{3\},\{1,4\},\{2,4\},\{3,4\}\}$, which gives $|\mathcal{A}||\mathcal{B}|=18$. They can be view as $\mathcal{A}=\mathcal{F}_{M_1,w_1}$ and $\mathcal{B}=\mathcal{F}_{M_2,w_2}$ for 
$$M_1=\begin{bmatrix}
1 & 1 & 1 & 0\\
0 & 0 & 0 & 1
\end{bmatrix},~~~w_1=\begin{bmatrix}1\\0\end{bmatrix},~~~M_2=\begin{bmatrix}1 & 1& 1 & 0\end{bmatrix},~~~w_2=\begin{bmatrix}1\end{bmatrix}.$$

For recovering pairs of families, it is expected that the following beautiful conjecture \cite{AS94} holds.
\begin{conjecture}[The Sandglass Conjecture]
If $(\mathcal{A}, \mathcal{B})$ is a recovering pair of families of subsets of $[n]$, then 
$$|\mathcal{A}||\mathcal{B}| \le 2^n.$$
\end{conjecture}
An equivalent form of the Sandglass Conjecture says if $(\mathcal{A}, \mathcal{B})$ is such that $A \cup B$ for $A \in \mathcal{A}, B \in \mathcal{B}$ uniquely determines $A$, $A \cap B$ uniquely determines $B$, then $|\mathcal{A}||\mathcal{B}| \le 2^n$. In our proof of Theorem \ref{thm_main}, we may introduce an additional random variable $V=X \cap Y$. Then $U$ determines $X$ and $V$ determines $Y$, and clearly $X, Y$ together determines $V$. This clearly gives new constraints $H[\cdots, U]=H[\cdots, U, X]$, $H[\cdots, V]=H[\cdots, V, Y]$, and $H[\cdots, X,Y]=H[\cdots, X, Y, V]$, etc. One can also define a new random variable $N$ which is a subset of $X \cap Y$ by keeping its elements with probability $2/3$. Then $X$ and $Y$ are again conditionally independent given $N$. We also have some flexibility to change the value of $K_{A,B}$ to $\lambda^{|A \cup B|}$ for some $\lambda \in (0,1)$, then the corresponding probability for $M$ and $N$ are $\lambda$ and $1-\lambda$ respectively. In our proof, we take $\lambda=1/3$ to optimize the bound. However, other choices of $\lambda$ may allow for additional entropy equalities and inequalities beyond those of Shearer type. It would be interesting to investigate further if our probabilistic technique can be extended to fully settle the Sandglass Conjecture.\\

\noindent {\bf Acknowledgment.} ChatGPT 5.6 was used to prove Lemma \ref{lem_key} and to explore and discover entropy inequalities in Theorem \ref{thm_main}, as well as proofreading and language polishing. All AI-assisted outputs were independently checked, verified, and simplified by the authors, who assume full responsibility for the content of this work. The authors thank Fan Chang for pointing out an issue with the definitions of \(\mu_{\mathrm{can}}\) and \(\mu_{\mathrm{rec}}\) in an earlier version of this paper.

\end{document}